\documentclass[12pt]{article}
\usepackage{amsmath,amssymb,color,amscd}
\usepackage{xspace}
\usepackage[matrix,arrow,curve]{xy}

\newcommand{\Var}{{\rm{Var}_{\mathbb{C}}}}

\def\1{\underline{1}}

\def\Z{{\mathbb Z}}

\def\C{{\mathbb C}}

\def\RR{{\mathcal R}}

\DeclareMathOperator{\spec}{Spec}

\def\Ind{{\rm Ind}}
\def\GG{{\mathcal{G}}}

\newtheorem{theorem}{Theorem}

\newtheorem{proposition}{Proposition}
\newtheorem{definition}{Definition}

\newenvironment{corollary}
{\smallskip\noindent{\bf Corollary\/}.}{\smallskip\par}

\newenvironment{remark}
{\smallskip\noindent{\bf Remark\/}.}{\smallskip\par}

\newenvironment{proof}
{\noindent{\bf Proof\/}.}{{ $\square$}\smallskip\par}

\title{Generalized orbifold Euler characteristics on the Grothendieck ring of varieties with actions of finite groups
}

\author{S.M.~Gusein-Zade \thanks{The work of the first author
(Sections~\ref{sec:fGr}, \ref{sec:Universal_Euler} and  \ref{sec:Instead_Macdonald}) 
was supported by the grant 16-11-10018 of the Russian Science Foundation.
Address: Moscow State University, Faculty
of Mechanics and Mathematics, GSP-1, Moscow, 119991, Russia. E-mail:
sabir\symbol{'100}mccme.ru} \and I.~Luengo \thanks{The last two authors were partially
supported by a competitive Spanish national grant MTM2016-76868-C2-1-P.
Address:  ICMAT (CSIC-UAM-UC3M-UCM), Dept. of Algebra, Geometry and Topology, Complutense University of Madrid,
Plaza de Ciencias 3, Madrid, 28040, Spain.
E-mail: iluengo\symbol{'100}mat.ucm.es} \and
A.~Melle-Hern\'andez \thanks{Address:  Instituto de Matem\'a¡tica Interdisciplinar (IMI),
Dept. of Algebra, Geometry and Topology, Complutense University of Madrid,
Plaza de Ciencias 3, Madrid, 28040, Spain. E-mail: amelle\symbol{'100}mat.ucm.es}}
\date{}
\begin{document}
\def\eps{\varepsilon}

\maketitle

\begin{abstract}
The notion of the orbifold Euler characteristic came from physics at the end of 80's. There were defined
higher order versions of the  orbifold Euler characteristic and generalized (``motivic'') versions of them.
In a previous paper the authors defined a notion of the Grothendieck ring $K_0^{\rm fGr}(\Var)$ of varieties
with actions of finite groups on which the orbifold Euler characteristic and its higher order versions are
homomorphisms to the ring of integers. Here we define  the generalized orbifold Euler characteristic and 
higher order versions of it as ring  homomorphisms from $K_0^{\rm fGr}(\Var)$ to the Grothendieck ring $K_0(\Var)$
of complex quasi-projective varieties and give some analogues of the classical 
Macdonald equations for the generating series of the Euler characteristics of the symmetric products of a space.
\end{abstract}

%%%%%%%%%%%%%%%%%%%%%%%%%%%%%%%%%%
\section{Introduction}\label{sec:Intro}
%%%%%%%%%%%%%%%%%%%%%%%%%%%%%%%%%%
The notion of the orbifold Euler characteristic $\chi^{\rm orb}$  came from physics at the 
end of 80's:,  \cite{DHVW}, see also \cite{AS} and \cite{HH}. Coincidence (up to sign) of the 
orbifold Euler characteristics is a necessary condition for orbifolds 
(or  rather for their crepant resolutions) to be mirror symmetric.
Higher order Euler characteristics $\chi^{(k)}$ of spaces with finite group actions were defined in    
\cite{AS} and \cite{BrF}.
The class of a variety in the Grothendieck ring $K_0(\Var)$ of complex quasi-projective varieties 
(being an additive invariant) can be considered as a generalized (``motivic'') Euler characteristic.
Generalized versions of the orbifold Euler characteristic and of
higher order Euler characteristics were defined in \cite{PSIM2007} (as a refinement of the so called
orbifold Hodge-Deligne polynomial,  see, e.~g., \cite{Wang})  and in \cite{GPh}.
(They take values in the extension of the Grothendieck $K_0(\Var)$ ring of complex quasi-projective varieties
by rational powers of the class of the affine line.)
One has the classical Macdonald equation for the generating series of the Euler characteristics of the 
symmetric products of a topological space.
Its versions for the orbifold Euler characteristic an for the higher order Euler characteristics were obtained in \cite{Tamanoi}.
Some versions for the generalized orbifold Euler characteristic an for the  generalized higher order
Euler characteristics were given in \cite{PSIM2007} and \cite{GPh}.
 
The orbifold Euler characteristic and the higher order Euler characteristics
are usually considered as functions on the Grothendieck ring  $K_0^G(\Var)$
of $G$-varieties. These functions define group homomorphisms from $K_0^G(\Var)$ to $\Z$, but not ring homomorphism.
A Grothendieck ring $K_0^{\rm fGr}(\Var)$ on which these Euler characteristics are ring homomorphisms 
was defined in \cite{PEMS}. It was called the Grothendieck ring  
of complex quasi-projective varieties with actions of finite groups.
In \cite{FAA2018}, there was defined a notion of the universal Euler characteristic on  
$K_0^{\rm fGr}(\Var)$. It takes values in a ring generated, as a free Abelian group, 
by elements corresponding to isomorphism classes of finite groups.

Here we define generalized orbifold Euler characteristic $\chi^{\rm orb}_{\rm g}$ and 
generalized higher order  Euler characteristics $\chi^{(k)}_{\rm g}$ as homomorphisms from 
$K_0^{\rm fGr}(\Var)$ to $K_0(\Var)$. We formulate Macdonald type equations for them in terms of 
$\lambda$-ring homomorphisms.
There is a map $\alpha:K_0^{\rm fGr}(\Var) \to  K_0^{\rm fGr}(\Var)$ such that 
$p\circ \alpha^{k}=\chi^{(k)}_{\rm g}$, where $p$ is the natural map  $K_0^{\rm fGr}(\Var) \to K_0(\Var)$
sending the class $[(X, G)]$ of a $G$-variety to the class $[X/G] $ of  its quotient.
We prove a substitute of a Macdonald type equation for the homomorphism $\alpha.$

%%%%%%%%%%%%%%%%%%%%%%%%%%%%%%%%%%
\section{Power structures and the Grothendieck ring of varieties with actions of finite groups}\label{sec:fGr}
%%%%%%%%%%%%%%%%%%%%%%%%%%%%%%%%%%
A power structure over a ring $R$ (commutative, with unit) is a method to give sense to an expression of the form
$\left(A(t)\right)^m$, where $A(t)=1+a_1t+a_2t^2+\ldots\in 1+ t\cdot R[[t]]$ and $m\in R$ (\cite{GLM-MRL}).
It is defined by a map
 $$
 \left(1+t\cdot R[[t]]\right)\times R\to 1+t\cdot R[[t]]\quad\quad
 ((A(t),m)\mapsto\left(A(t)\right)^m)
 $$
 which satisfies the following properties:
 \begin{enumerate}
\item[1)]  $\left(A(t)\right)^0=1$;
\item[2)]  $\left(A(t)\right)^1=A(t)$;
\item[3)]  $\left(A(t)\cdot B(t)\right)^{m}=\left(A(t)\right)^{m}\cdot \left(B(t)\right)^{m}$;
\item[4)]  $\left(A(t)\right)^{m+n}=\left(A(t)\right)^{m}\cdot \left(A(t)\right)^{n}$;
\item[5)]  $\left(A(t)\right)^{mn}=\left(\left(A(t)\right)^{n}\right)^{m}$;
\item[6)]  $(1+a_1t+\ldots)^m=1+ma_1t+\ldots$;
\item[7)]  $\left(A(t^k)\right)^m =
\left(A(t)\right)^m\raisebox{-0.5ex}{$\vert$}{}_{t\mapsto t^k}$ for $k\in\Z_{>0}$.
 \end{enumerate}

Power structures over a ring are related with $\lambda$-structures on it. A $\lambda$-structure on a ring $R$
(sometimes called a pre-lambda structure: see, e.~g., \cite{Knutson}) is an additive-to-multiplicative homomorphism
$R\to 1+t\cdot R[[t]]$ (that is $a\mapsto \lambda_a(t)$, $\lambda_{a+b}(t)=\lambda_a(t)\cdot\lambda_b(t)$) such that
$\lambda_a(t)=1+at+\ldots$ A $\lambda$-structure on $R$
defines a power structure over it in the following way. A series $A(t)\in 1+ t\cdot R[[t]]$ can be in a unique way
represented as the product $\prod\limits_{k=1}^{\infty}\lambda_{b_k}(t^k)$. Then one defines $\left(A(t)\right)^m$
as $\prod\limits_{k=1}^{\infty}\lambda_{mb_k}(t^k)$. A power structure over $R$ permits to define a number of
$\lambda$-structures on it: for any series $\lambda_1(t)=1+t+b_2t^2+\ldots$ one can put
$\lambda_a(t)=\left(\lambda_1(t)\right)^a$.

%% Opposite structure???
The standard power structure over the ring $\Z$ of integers is defined by the standard exponent of a series.
A natural power structure over the Grothendieck ring $K_0(\Var)$ of complex quasi-projective varieties was
introduced in \cite{GLM-MRL}. It is defined by the formula
\begin{eqnarray}\label{eqn:geometric}
 &\ &(1+[A_1]t+[A_2]t^2+\ldots)^{[M]}=\label{Power}\\
 &=&1+\sum_{k=1}^{\infty}
 \left(\sum_{\{k_i\}:\sum ik_i=k}
 \left[\left(\left(M^{\sum_i k_i}\setminus\Delta\right)\times\prod_i A_i^{k_i}\right)\left/
 {\prod_i S_{k_i}}\right.\right]\right)\cdot t^k,\nonumber
\end{eqnarray}
where $A_i$, $i=1, 2, \ldots$, and $M$ are complex quasi-projective varieties, $\Delta$ is the big diagonal in
$M^{\sum_i k_i}$,
%% that is, the set of (ordered) $\left(\sum_i k_i\right)$-tuples of points of $M$ with at least
%% two coinciding ones,
the group $S_{k_i}$ acts by the simultaneous permutations on the components of $M^{k_i}$ and on the components
of $A_i^{k_i}$.

For a topological space $X$ (say, a complex quasi-projective variety) with an action of a finite group $G$, one has the
notions of the orbifold Euler characteristic $\chi^{\rm orb}(X,G)$ and of the (orbifold) Euler characteristics
$\chi^{(k)}(X,G)$ of higher orders (see, e.~g., \cite{AS}, \cite{HH}, \cite{BrF}). They can be defined, in particular,
in the following way. Let $\chi^{(0)}(X,G):=\chi(X/G)$, where $\chi$ is the (additive) Euler characteristic defined
through cohomologies with compact support. For $k\ge 1$, let
$$
\chi^{(k)}(X,G):=\sum_{[g]\in{\rm Conj\,}G}\chi^{(k-1)}(X^{\langle g\rangle},C_G(g))\,,
$$
where ${\rm Conj\,}G$ is the set of conjugacy classes of elements of $G$, $g$ is a representative of the class $[g]$,
$X^{\langle g\rangle}$ is the fixed point set of $g$, $C_G(g)$ is the centralizer of the element $g$ in $G$.
The orbifold Euler characteristic $\chi^{\rm orb}(X,G)$ is the Euler characteristic of order $1$:
$\chi^{(1)}(X,G)$.

The orbifold Euler characteristic and  the Euler characteristics of higher orders can be considered as
functions on the Grothendieck ring $K_0^G(\Var)$ of quasi-projective $G$-varieties. These functions are group
homomorphisms, but not ring ones. A ring on which they are defined as ring homomorphisms to $\Z$ was introduced
in \cite{PEMS}.

Let us consider $G$-varieties, i.~e.\ pairs $(X,G)$ consisting of a complex quasi-projective variety $X$ and a finite
group $G$ acting on $X$. We shall call two pairs $(X,G)$ and $(X',G')$ {\em isomorphic} if there exist an isomorphism
$\psi:X\to X'$ of quasi-projective varieties and a group isomorphism $\varphi:G\to G'$ such that
$\psi(gx)=\varphi(g)\psi(x)$ for $x\in X$, $g\in G$. If $G$ is a subgroup of a finite group $H$, one has the
induction operation ${\rm Ind}_G^H$ which converts $G$-varieties to $H$-varieties. For a $G$-variety $X$,
${\rm Ind}_G^H X$ is the quotient of $H\times X$ by the right action of the group $G$ defined by
$(h,x)*g=(hg, g^{-1}x)$. (The action of $H$ on ${\rm Ind}_G^H X$ is defined in the natural way: $h_0(h,x)=(h_0h,x)$.)

\begin{definition}\label{def:Groth_with_actions} (see \cite{PEMS})
The Grothendieck ring of complex quasi-projective varieties with actions of finite groups is the abelian group
$K_0^{\rm fGr}(\Var)$ generated by the classes $[(X,G)]$ of $G$-varieties (for different finite groups $G$)
modulo the relations:
\begin{enumerate}
 \item[1)] if $(X,G)$ and $(X',G')$ are isomorphic, then $[(X,G)]=[(X',G')]$;
 \item[2)] if $Y$ is a Zariski closed $G$-subvariety of a $G$-variety $X$, then $[(X,G)]=[(Y,G)]+[(X\setminus Y,G)]$;
 \item[3)] if $(X,G)$ is a $G$-variety and $G$ is a subgroup of a finite group $H$, then
 $[(\Ind_G^H ,H)]=[(X,G)]$.
\end{enumerate}
The multiplication in $K_0^{\rm fGr}(\Var)$ is defined by the Cartesian product:
$$
[(X_1,G_1)]\times[(X_2,G_2)]=[(X_1\times X_2,G_1\times G_2)]\,.
$$
\end{definition}

The unit element in $K_0^{\rm fGr}(\Var)$ is $1=[ (\spec(\, \C\, ), (e))]$, the class of the one-point variety
with the action of the group with one element.

\begin{remark}
 This ring (under the name ``the Grothendieck ring of equivariant varieties'') was used in \cite{BGLL}.
\end{remark}

One has a natural ring homomorphism $p:K_0^{\rm fGr}(\Var)\to K_0(\Var)$ sending $[(X,G)]$ to $[X/G]$.

There are two (``geometric'') $\lambda$-structures on the ring $K_0^{\rm fGr}(\Var)$.
Let $X$ be a $G$-variety. The Cartesian power $X^n$ carries natural actions of the group $G^n$ (acting
component-wise) and of the group $S_n$ (acting by permuting the factors in $X^n$) and therefore an action
of their semi-direct product (the wreath-product) $G^n\rtimes S_n=G_n$.

\begin{definition}
 The Kapranov zeta function of $(X,G)$ is
 $$
 \zeta_{(X,G)}(t)=1+\sum_{n=1}^{\infty}[(X^n, G_n)]t^n\in 1+t\cdot K_0^{\rm fGr}(\Var)[[t]].
 $$
\end{definition}

In \cite{PEMS}, it is shown that the Kapranov zeta function is well-defined for elements of the ring
$K_0^{\rm fGr}(\Var)$ and defines a $\lambda$-ring structure on it.

Another $\lambda$-structure on the ring $K_0^{\rm fGr}(\Var)$ is defined by the generating series
of classes of equivariant configuration spaces of points in $X$. Let $\Delta_G$ be the big $G$-diagonal in the
Cartesian power $X^n$ of a $G$-variety $X$, i.~e.\ the set of $n$-tuples $(x_1,\ldots, x_n)\in X^n$ with at least
two of $x_i$ from the same $G$-orbit. The wreath product $G_n$ acts on $X^n\setminus \Delta_G$. Let
$$
\lambda_{(X,G)}(t)=1+\sum_{n=1}^{\infty}[(X^n\setminus \Delta_G, G_n)]t^n\in 1+t\cdot K_0^{\rm fGr}(\Var)[[t]]
$$
be the generating series of classes of equivariant configuration spaces of points in $X$.
In \cite{PEMS}, it is shown that the series $\lambda_{(X,G)}(t)$ defines a $\lambda$-structure on the ring
$K_0^{\rm fGr}(\Var)$ and there was given a geometric description of the power structure over $K_0^{\rm fGr}(\Var)$
corresponding to this $\lambda$-structure. (A geometric description of the power structure over the ring
$K_0^{\rm fGr}(\Var)$ corresponding to the $\lambda$-structure defined by the Kapranov zeta function is not known.)
We shall call these $\lambda$-structures (and the corresponding power structures) the symmetric product and
the configuration space ones.

In \cite{PEMS}, it was shown that the orbifold Euler characteristic and the higher order Euler characteristics
of an element of $K_0^{\rm fGr}(\Var)$ are well-defined (that is $\chi^{(k)}(\Ind_G^H X,H)=\chi^{(k)}(X,G)$) and
they are ring homomorphisms from $K_0^{\rm fGr}(\Var)$ to $\Z$.

One has a ring homomorphism $\alpha: K_0^{\rm fGr}(\Var)\to K_0^{\rm fGr}(\Var)$ defined by
$\alpha([(X,G)])=\sum_{[g]\in{\rm Conj\,}G}[(X^{\langle g\rangle}, C_G(g))]$ (see the notations above).
One can see that $\chi^{(k)}=\chi\circ p\circ\alpha^k$, where $\chi:K_0(\Var)\to\Z$ is the usual Euler characteristic.
Therefore $\alpha^k$ can be considered as a sort of a generalized version of the Euler characteristic of order $k$
with values in $K_0^{\rm fGr}(\Var)$.
(In \cite{BGLL} the homomorphism $\alpha$ is called the inertia homomorphism.)

%%%%%%%%%%%%%%%%%%%%%%%%%%%%%%%%%%
\section{The universal Euler characteristic}\label{sec:Universal_Euler}
%%%%%%%%%%%%%%%%%%%%%%%%%%%%%%%%%%
In~\cite{FAA2018}, there was defined the so-called {\em universal Euler characteristic} on the ring
$K_0^{\rm fGr}(\Var)$. Let $\RR$ be the subring of $K_0^{\rm fGr}(\Var)$ generated by the zero-dimensional
(i.~e.\ finite) $G$-varieties. It can be described in the following way. Let $\GG$ be the set of isomorphisms
classes of finite groups. Then $\RR$ is the Abelian group freely generated by the elements $T^{\mathfrak G}$
corresponding to the isomorphism classes ${\mathfrak G}\in\GG$ of finite groups. The generator $T^{\mathfrak G}$
is represented by the one-point set with the (unique) action of a representative $G$ of the class ${\mathfrak G}$.
The Krull-Schmidt theorem implies that $\RR$ is the ring of polynomials in the variables $T^{\mathfrak G}$
corresponding to the isomorphism classes of finite indecomposable groups. If $(X,G)$ is a $G$-variety, its universal
Euler characteristic is defined by
$$
\chi^{\rm un}(X,G):= \sum_{{\mathfrak H}\in\GG}\chi\left(X^{({\mathfrak H})}/G\right)\cdot T^{\mathfrak H},
$$
where $X^{({\mathfrak H})}$ is the set of points $x\in X$ with the isotropy subgroup $G_x=\{g\in G: gx=x\}$
belonging to the class $\mathfrak H$.

\begin{remark}
 This characteristic can be regarded as a universal one in the topological category.
\end{remark}

The orbifold Euler characteristic and the higher order Euler characteristics define ring homomorphisms from $\RR$
to $\Z$.

One has a natural $\lambda$-ring structure on $\RR$ defined by an analogue of the Kapranov zeta function (see \cite{FAA2018})
and the maps $\chi^{(k)}$ are $\lambda$-ring homomorphisms with respect to this $\lambda$-structure.

One has the commutative diagram of ring homomorphisms
$$
\xymatrix{
K_0^{fGr}(\Var)\ar[rr]^{\ \ \ \chi^{\rm un}}\ar[d]_{\alpha^k}\ar[rd]^{\chi^{(k)}}
&& \RR\ar[d]^{\alpha^k}\ar[ld]_{\ \ \chi^{(k)}}\\
K_0^{fGr}(\Var)\ar[r]^{\ \ \ \chi^{(0)}}& \Z & \RR.\ar[l]_{\ \ \chi^{(0)}}
}
$$

%%%%%%%%%%%%%%%%%%%%%%%%%%%%%%%%%%
\section{Macdonald type equations and $\lambda$-structure homomorphisms}\label{sec:Macdonald}
%%%%%%%%%%%%%%%%%%%%%%%%%%%%%%%%%%
The classical Macdonald equation describes the generating series of the Euler characteristics of the symmetric
products of a space:
\begin{equation}\label{eqn:classical_Macdonald}
 1+\sum_{n=1}^{\infty}\chi(S^nX)\cdot t^n=(1-t)^{-\chi(X)},
\end{equation}
where $X$ is a topological space and $S^nX=X^n/S_n$ is its $n$th symmetric product. One also has a Macdonald type
equation for the generating series of the Euler characteristics of the configuration spaces of subsets of points
in $X$. Let $M_nX=(X^n\setminus \Delta)/S_n$ be the configuration space of unordered $n$-tuples of points in $X$,
where $\Delta$ is the big diagonal in $X^n$. Then one has
\begin{equation}\label{eqn:Macdonald_for_conf}
 1+\sum_{n=1}^{\infty}\chi(M_nX)\cdot t^n=(1+t)^{\chi(X)}.
\end{equation}
There exist equations for the generating series of the Hodge-Deligne polynomials of the symmetric products
of a complex quasi-projective variety and of the configuration spaces of subsets of points in it
(see, e.~g., \cite{PSIM2007}).

These equations are related with $\lambda$-ring homomorphisms (and therefore with power structure homomorphisms)
from the Grothendieck ring $K_0(\Var)$ of complex quasi-projective varieties to $\Z$ and to $\Z[u,v]$ respectively.
For a ring homomorphism $\varphi:R_1\to R_2$, one has a natural map (a group homomorphism)
$\varphi_*:1+t\cdot R_1[[t]]\to 1+t\cdot R_2[[t]]$ obtained by applying $\varphi$ to the coefficients of a series.
If $R_1$ and $R_2$ are $\lambda$-rings, a ring homomorphism $\varphi:R_1\to R_2$ is said to be a $\lambda$-ring
homomorphism if $\lambda_{\varphi(a)}(t)=\varphi_*\lambda_a(t)$ for $a\in R_1$. If $R_1$ and $R_2$ are rings
with power structures, $\varphi:R_1\to R_2$ is a power structure homomorphism if
$\varphi_*\left(A(t)^m\right)=\left(\varphi_*(A(t))\right)^{\varphi(m)}$. A $\lambda$-ring homomorphism induces
a power structure homomorphism for the corresponding power structures and vise-versa.

Equations~(\ref{eqn:classical_Macdonald}) and~(\ref{eqn:Macdonald_for_conf}) are related with the following
$\lambda$-ring structures on the Grothendieck ring $K_0(\Var)$ and on $\Z$. The Kapranov zeta function of a
variety $X$ (or of its class $[X]\in K_0(\Var)$) is
$$
\zeta_{X}(t)=1+\sum_{n=1}^{\infty}[S^nX]\cdot t^n=(1-t)^{-[X]}.
$$
One can see that $\zeta_{X}(t)$ defines a $\lambda$-ring structure on $K_0(\Var)$. Another $\lambda$-ring
structure on $K_0(\Var)$ is defined by the series
$$
\lambda_{X}(t)=1+\sum_{n=1}^{\infty}[M_nX]\cdot t^n=(1+t)^{[X]}.
$$
The corresponding $\lambda$-structures on $\Z$ are $\check{\zeta}_n(t)=(1-t)^{-n}$ and $\check{\lambda}_n(t)=(1+t)^n$
respectively. Equations~(\ref{eqn:classical_Macdonald}) and (\ref{eqn:Macdonald_for_conf}) mean that the Euler
characteristic is a $\lambda$-ring homomorphism for the corresponding $\lambda$-structures.
The both $\lambda$-structures on $K_0(\Var)$ (as well as the both $\lambda$-structures on $\Z$)
induce the same power structure over $K_0(\Var)$ (over $\Z$ respectively). The corresponding power structure
over $K_0(\Var)$ is given by Equation~(\ref{eqn:geometric}).

The Macdonald type equations for the orbifold Euler characteristic and for the higher order Euler characteristics
look like following.

\begin{theorem}
 The map $\chi^{(k)}:K_0^{\rm fGr}(\Var)\to \Z$ is a $\lambda$-ring homomorphism
 for the $\lambda$-structures on the source $K_0^{\rm fGr}(\Var)$ and on the target $\Z$
 defined by the Kapranov zeta function $\zeta_{(X,G)}(t)$ and by the series
 $$
 \lambda^{(k)}_{n}(t)=
 \left(\prod_{r_1,\ldots, r_k\ge 1}
 \left(1-t^{r_1\cdot\ldots\cdot r_k}\right)^{r_2\cdot r_3^2\cdot\ldots\cdot r_k^{k-1}}\right)^{n}
 $$
 respectively, i.~e.
 $$
 \chi^{(k)}\left(\zeta_{(X,G)}(t)\right)=
 \lambda^{(k)}_{\chi^{(k)}(X,G)}(t)\,.
 $$
\end{theorem}

\begin{remark}
 The map $\chi^{(k)}:K_0^{\rm fGr}(\Var)\to\Z$ is not a $\lambda$-ring homomorphism
 with respect to the configuration space $\lambda$-structure on $K_0^{\rm fGr}(\Var)$.
\end{remark}

%%%%%%%%%%%%%%%%%%%%%%%%%%%%%%%%%%
\section{Generalized Euler characteristics of higher orders as homomorphisms from
$K_0^{\rm fGr}(\Var)$}\label{sec:Generalized-higher}
%%%%%%%%%%%%%%%%%%%%%%%%%%%%%%%%%%
In \cite{PEMS}, it was shown that the orbifold Euler characteristic and the Euler characteristics of higher orders
are ring homomorphism (moreover, $\lambda$-ring homomorphisms) from the Grothendieck ring $K_0^{\rm fGr}(\Var)$
to $\Z$. The notions of generalized (``motivic'') orbifold Euler characteristic and  generalized (``motivic'') 
Euler characteristic of higher orders were introduced by the authors in \cite{PSIM2007} (as a refinement of the
orbifold Hodge-Deligne polynomial from \cite{Wang}) and in \cite{GPh}. It was defined as an invariant of a
complex quasi-projective manifold with the action of a finite group  and took values in the extension of the
Grothendieck ring $K_0(\Var)$ of  complex quasi-projective varieties by the rational powers of the class of the
affine line and was not defined on a ring.
In \cite{PEMS}, these invariants were considered as functions on the Grothendieck ring of varieties with equivariant
vector bundles. Here we define versions of them (with the corresponding weight  $\underline{\varphi}=0$ in terms
of \cite{GPh} and \cite{PEMS}) as ring homomorphisms from $K_0^{\rm fGr}(\Var)$ to $K_0(\Var)$.

The homomorphism $p:K_0^{\rm fGr}(\Var) \to K_0(\Var)$ sending the class $[(X, G)]$ to the class $[X/G$] is an
additive function on $K_0^{\rm fGr}(\Var)$ and therefore can be considered as a generalized version of the Euler
chararteristic. Let us call it \emph{generalized Euler characteristic of order} $0$ and denote by  
$\chi_{\rm g}^{(0)}(X,G)$. Let $X$ be a G-variety.

\begin{definition}\label{def:Generalized_Euler}
The \emph{generalized Euler characteristic} $\chi_{\rm g}^{(k)}(X,G)$ \emph{of order} $k$  is defined by 
$$
\chi_{\rm g}^{(k)}(X,G)= \sum_{[g]\in{\rm Conj\,}G} \chi_{\rm g}^{(k-1)}(X^{\langle g\rangle},C_{G}(g))\in K_0(\Var),
$$
where the sum is over the conjugacy classes $[g]$ of elements of $G$, $g$ is a representative of the class $[g]$, 
$X^{\langle g\rangle}$ is the fixed point set of $g$ and $C_{G}(g)$ is the centralizer of $g$ in $G$.
\end{definition}

\begin{proposition}
 The generalized Euler characteristic $\chi_{\rm g}^{(k)}$ of order $k$  is a well defined ring homomorphism  
from $K_0^{\rm fGr}(\Var)$ to $K_0(\Var)$.
\end{proposition}

\begin{proof}
 One has to show that $\chi_{\rm g}^{(k)}$ respects the relations 1)-3) of Definition~\ref{def:Groth_with_actions}.
 This obviusly holds for 1) and 2).
 The fact that $\chi_{\rm g}^{(k)}$ respects the condition 3) (the induction relation) can be proved by
 induction on $k$: it is obvious for $k=0$ and the statement for an arbitrary $k$ follows from the statement
 for $k-1$ due to \cite[Lemma 1]{PEMS}.
\end{proof}

 The following statement is a specification of \cite[Theorem 4]{PEMS} or of \cite[Theorem 1]{GPh}.

\begin{theorem}
 The map $\chi^{(k)}_{\rm g}:K_0^{\rm fGr}(\Var)\to K_0(\Var)$ is a $\lambda$-ring homomorphism
 for the $\lambda$-structures on the source $K_0^{\rm fGr}(\Var)$ and on the target $K_0(\Var)$
 defined by the Kapranov zeta function $\zeta_{(X,G)}(t)$ and by the series
 $$
 \lambda^{(k)}_{X}(t)=
 \left(\prod_{r_1,\ldots, r_k\ge 1}
 \left(1-t^{r_1\cdot\ldots\cdot r_k}\right)^{r_2\cdot r_3^2\cdot\ldots\cdot r_k^{k-1}}\right)^{[X]},
 $$
 i.~e.
 $$
 \chi^{(k)}_{\rm g}\left(\zeta_{(X,G)}(t)\right)=
 \lambda^{(k)}_{\chi^{(k)}_{\rm g}(X,G)}(t)\,.
 $$
\end{theorem}

\begin{remark}
 The map $\chi^{(k)}_{\rm g}$ is not a $\lambda$-ring homomorphism
 with respect to the configuration space $\lambda$-structure on $K_0^{\rm fGr}(\Var)$.
\end{remark}

One has the commutative diagram of ring homomorphisms
$$
\xymatrix{
&K_0^{fGr}(\Var)\ar[d]^{\chi^{(k)}_{\rm g}}\ar[ld]_{\alpha^k}\ar[rd]^{\chi^{(k)}}&\\
K_0^{fGr}(\Var)\ar[r]^{\ \ p}& K_0(\Var)\ar[r]^{\ \ \chi}&\Z.
}
$$

%%%%%%%%%%%%%%%%%%%%%%%%%%%%%%%%%%
\section{A substitute of a Macdonald type equation for the homomorphism $\alpha$}\label{sec:Instead_Macdonald}
%%%%%%%%%%%%%%%%%%%%%%%%%%%%%%%%%%
As it follows from Definition~\ref{def:Generalized_Euler}, the composition of the homomorphism $\alpha^k$ with
the natural map
$p:K_0^{\rm fGr}(\Var)\to K_0(\Var)$ coincides with the generalized (``motivic'') Euler characteristic of order $k$
(with the generalized orbifold Euler characteristic for $k=1$) computed without the fermion shift (i.~e.\ with
$\underline{\varphi}=0$ in terms of \cite{GPh}). Its composition with the usual Euler characteristic homomorphism
$\chi:K_0(\Var)\to\Z$ gives the usual (orbifold) Euler characteristic of order $k$. One has Macdonald type equations
for the (orbifold) Euler characteristic of order $k$ (\cite{Tamanoi}) and for the generalized Euler characteristic
of order $k$ (\cite{GPh}). Here we shall give a version of these equations for the homomorphism $\alpha$ (which
reduces to the Macdonald type equations for the generalized orbifold Euler characteristic and for the orbifold
Euler characteristic after applying the homomorphisms $p$ and $\chi\circ p$ respectively). Let
$\alpha_r:K_0^{\rm fGr}(\Var)\to K_0^{\rm fGr}(\Var)$ be defined in the following way. Let $(X,G)$ be a $G$-variety.
For a representative $g$ of a conjugacy class $[g]\in {\rm Conj\,}G$, the centralizer $C_G(g)$ acts on the fixed
point set $X^{\langle g\rangle}$ of the element $g$. Let $C_G(g)\langle a_{r,g}\rangle$ be the group generated by
$C_G(g)$ and by an additional element $a_{r,g}$ commuting with all the elements of $C_G(g)$ and such that
$a_{r,g}^r=g$. For $r=1$, the group $C_G(g)\langle a_{r,g}\rangle$ coincides with $C_G(g)$. One has an action
of the group $C_G(g)\langle a_{r,g}\rangle$ on $X^{\langle g\rangle}$ assuming $a_{r,g}$ to act trivially.
Let $\alpha_r:K_0^{\rm fGr}(\Var)\to K_0^{\rm fGr}(\Var)$ be defined by
$$
\alpha_r\left([(X,G)]\right):=\sum_{[g]\in {\rm Conj\,}G}[(X^{\langle g\rangle},C_G(g)\langle a_{r,g}\rangle)]\,.
$$
In particular $\alpha_1=\alpha$.

\begin{theorem}
\begin{equation}
 \alpha\left(\zeta_{(X,G)}(t)\right)=\prod_{r=1}^{\infty} \zeta_{\alpha_r\left([(X,G)]\right)}(t^r)\,.
\end{equation}
\end{theorem}

\begin{proof}
 The proof essentially follows from the description of the conjugacy classes of elements of the wreath products
 $G_n$ and of their centralizers from, e.~g., \cite{Tamanoi} and to a big extent repeats the computations in
 \cite{GPh}. An element of $G_n$ can be written as a pair $({\bf g},s)$, where ${\bf g}=(g_1,\ldots,g_n)\in G^n$,
 $s\in S_n$. One has
 \begin{eqnarray}
 &\ &\alpha\left(\zeta_{(X,G)}(t)\right)=\\
 &=&\sum_{n\ge 0}t^n\cdot\left(\sum_{[({\bf g},s)]\in {\rm Conj\,}G_n}
 \left[(X^n)^{\langle({\bf g},s)\rangle}, C_{G_n}\left( ({\bf g},s)\right)\right]\right)\,,
 \end{eqnarray}
 where the sum is over the conjugacy classes $[({\bf g},s)]$ of elements of $G_n$. The conjugacy classes
 $[({\bf g},s)]$ of elements $({\bf g},s)=(g_1,\ldots,g_n;s)$ of $G_n$ are characterized by their types.
 For a cycle $z=(i_1,\ldots,i_r)$ (of length $r$) in the permutation $s$, its {\em cicle-product}
 $g_{i_r}g_{i_{r-1}}\ldots g_{i_1}$ is well-defined up to conjugacy. For $[c]\in{\rm Conj\,}G$ and for $r\ge 1$,
 let $m_r(c)$ be the number of $r$-cycles in $s$ with the cicle-product from $[c]$. One has
 $$
 \sum_{[c]\in{\rm Conj\,}G,\ r\ge1}rm_r(c)=n\,.
 $$
 The collection $\{m_r(c)\}_{r,[c]}$ is called the type of the element $({\bf g},s)\in G_n$. Two elements
 of $G_n$ are conjugate if and only if they are of the same type. Therefore the summation over the conjugacy
 classes of elements of $G_n$ can be substituted by the summation over all possible types.

The fixed point set $(X^n)^{\langle ({\bf g},s)\rangle }$ can be identified with 
\begin{equation*}
\prod_{[c]\in {\rm Cong\,}G,r\ge 1}(X^{\langle c\rangle})^{m_r(c)}\,.
\end{equation*}
The centralizer of $({\bf g},s)\in G_n$ is isomorphic to
\begin{equation*}
\prod_{[c]\in {\rm Cong\,}G,r\ge 1}\left(C_G(c)\langle a_{r,c}\rangle\right)_{m_r(c)}\,,
\end{equation*}
where the definition of the group $C_G(c)\langle a_{r,c}\rangle$ and the description of its action on the fixed
point set $X^{\langle c\rangle}$ are given above.

Therefore one has
\begin{eqnarray*}
{}&\ & \alpha\left(\zeta_{(X,G)}(t)\right)=\sum_{n=0}^{\infty}t^n\cdot
\left(
\sum_{[({\bf g},s)]\in {\rm Cong\,}G_n}
\left[(X^{\langle ({\bf g},s)\rangle}, C_{G_n}\left(({\bf g},s)\right) \right]\right)\\
{}&=&\sum_{n=0}^{\infty}t^n\cdot
\left(\sum_{\{m_r(c)\}}\left[\prod_{[c], r}
(X^{\langle c\rangle })^{m_r(c)},\prod_{[c],r}(C_G(c)\langle a_{r,c}\rangle)_{m_r(c)}\right]\right)\\
&{\ =}&\sum_{\{m_r(c)\}}t^{\sum rm_r(c)}\cdot\prod_{[c],r} 
\left[\left(X^{\langle c\rangle}\right)^{m_r(c)}, (C_G(c)\langle a_{r,c}\rangle)_{m_r(c)}\right]\\
&{\ =}&\prod_{r=1}^{\infty}\prod_{[c]}\sum_{m_r(c)=1}^{\infty}\left(t^{rm_r(c)}\left[(X^{\langle c\rangle})^{m_r(c)},
(C_G(c)\langle a_{r,c}\rangle)_{m_r(c)}\right]
\right)\\
&=& \prod_{r=1}^{\infty}\prod_{[c]}\zeta_{(X^{\langle c\rangle},C_G(c)\langle a_{r,c}\rangle)}(t^r)=
\prod_{r=1}^{\infty}\zeta_{\sum_{[c]}[(X^{\langle c\rangle},C_G(c)\langle a_{r,c}\rangle)]}(t^r)\\
&=& \prod_{r=1}^{\infty}\zeta_{\alpha_r([(X, G)])}(t^r).
\end{eqnarray*}
\end{proof}

The restrictions of the homomorphisms $\alpha$ and $\alpha_r$ to the subring $\RR\subset K_0^{\rm fGr}(\Var)$ 
define the homomorphisms  $\alpha$ and $\alpha_r$ from $\RR$ to $\RR.$ The homomorphism $\alpha_r$ acts by the
formula $\alpha_r(T^{[G]})=\sum_{[g]\in {\rm Conj\, G}} T^{[C_G(g)\langle a_{r,g}\rangle]}$.

\begin{corollary} For $a\in \RR$ one has
\begin{equation}
 \alpha\left(\zeta_{a}(t)\right)=\prod_{r=1}^{\infty} \zeta_{\alpha_r\left(a\right)}(t^r)\,.
\end{equation}
\end{corollary}

\end{document}